\documentclass[twoside,leqno,10pt,A4]{amsart}
\usepackage{amsfonts}
\usepackage{amsmath}
\usepackage{amscd}
\usepackage{amssymb}
\usepackage{amsthm}
\usepackage{amsrefs}
\usepackage{latexsym}
\usepackage{bbm}
\usepackage{color}
\begin{document}

\newtheorem{theorem}[subsection]{Theorem}
\newtheorem{proposition}[subsection]{Proposition}
\newtheorem{lemma}[subsection]{Lemma}
\newtheorem{corollary}[subsection]{Corollary}
\newtheorem{conjecture}[subsection]{Conjecture}
\newtheorem{prop}[subsection]{Proposition}
\numberwithin{equation}{section}
\newcommand{\mr}{\ensuremath{\mathbb R}}
\newcommand{\dif}{\mathrm{d}}
\newcommand{\intz}{\mathbb{Z}}
\newcommand{\ratq}{\mathbb{Q}}
\newcommand{\natn}{\mathbb{N}}
\newcommand{\comc}{\mathbb{C}}
\newcommand{\rear}{\mathbb{R}}
\newcommand{\prip}{\mathbb{P}}
\newcommand{\uph}{\mathbb{H}}
\newcommand{\fief}{\mathbb{F}}
\newcommand{\majorarc}{\mathfrak{M}}
\newcommand{\minorarc}{\mathfrak{m}}
\newcommand{\sings}{\mathfrak{S}}
\newcommand{\fA}{\ensuremath{\mathfrak A}}
\newcommand{\mn}{\ensuremath{\mathbb N}}
\newcommand{\mq}{\ensuremath{\mathbb Q}}
\newcommand{\half}{\tfrac{1}{2}}
\newcommand{\f}{f\times \chi}
\newcommand{\summ}{\mathop{{\sum}^{\star}}}
\newcommand{\chiq}{\chi \bmod q}
\newcommand{\chidb}{\chi \bmod db}
\newcommand{\chid}{\chi \bmod d}
\newcommand{\sym}{\text{sym}^2}
\newcommand{\hhalf}{\tfrac{1}{2}}
\newcommand{\sumstar}{\sideset{}{^*}\sum}
\newcommand{\sumprime}{\sideset{}{'}\sum}
\newcommand{\sumprimeprime}{\sideset{}{''}\sum}
\newcommand{\shortmod}{\ensuremath{\negthickspace \negthickspace \negthickspace \pmod}}
\newcommand{\V}{V\left(\frac{nm}{q^2}\right)}
\newcommand{\sumi}{\mathop{{\sum}^{\dagger}}}
\newcommand{\mz}{\ensuremath{\mathbb Z}}
\newcommand{\leg}[2]{\left(\frac{#1}{#2}\right)}
\newcommand{\muK}{\mu_{\omega}}

\title{Primes $p \equiv 1 \bmod{d}$ and $a^{(p-1)/d} \equiv 1 \bmod{p}$}
\date{\today}
\author{Peng Gao and Liangyi Zhao}
\maketitle

\begin{abstract}
Suppose that $d \in \{ 2, 3, 4, 6 \}$ and $a \in \intz$ with $a\neq -1$ and $a$ is not square.  Let $P_{(a,d)}$ be the number of primes $p$ not exceeding $x$ such that $p \equiv 1 \pmod{d}$ and $a^{(p-1)/d} \equiv 1 \pmod{p}$.  In this paper, we study the mean value of $P_{(a,d)}$.
\end{abstract}

\noindent {\bf Mathematics Subject Classification (2010)}: 11A07 \newline

\section{Introduction}

Let $\prip$ denote the set of prime numbers, $d>0$ be a squarefree integer and $a \in \intz$ with $a \neq -1$ and $a$ is not a square.  Set
\[ \mathcal{P}_{(a,d)} = \{ p \in \prip : p \equiv 1 \pmod{d} \; \;  \; \mbox{and} \; \; \; a^{(p-1)/d} \equiv 1 \pmod{p} \}  \]
and
\[ P_{a,d} (x) = \# \{ p \in \mathcal{P}_{(a,d)} : p \leq x \} . \]
It has long been known and studied \cites{murty, Hoo} that the estimation of $P_{(a,d)}$ is connected with Artin's conjecture on primitive roots, which asserts that the number of primes not exceeding $x$ and for which $a$ is a primitive root is asymptotically $c \pi(x)$ where $c$ is some explicitly defined constant whose choice depends on the value of $a$. \newline

Assuming that truth of the Riemann hypothesis for Dedekind zeta function for the number field $\ratq[\sqrt[d]{a}, \sqrt[d]{1}]$, C. Hooley \cite{Hoo} proved that
\begin{equation} \label{hoobound}
 P_{(a,d)}(x) = \frac{\varepsilon(d)}{d\varphi(d)} \mathrm{li}(x) + O \left( \sqrt{x} \log (dx) \right) .
 \end{equation}
where $\varepsilon(d)$ is defined in the following way.  Write $a=lm^2$ where $l$ is square free.  $\varepsilon(d)=2$ if $l \equiv 1 \pmod{4}$ and $2l |d$ and $\varepsilon(d)=1$ otherwise.  Using \eqref{hoobound}, Hooley was able to prove Artin's conjecture conditionally.   For a survey of this conjecture, we refer the reader to \cite{murty}. \newline

S. Li \cite{Li} studied $P_{(a,d)}(x)$ on average and proved that
\begin{equation} \label{Libound}
\frac{1}{y} \sum_{2 \leq a \leq y} P_{(a,d)}(x) = \frac{\pi(x;1,d)}{d} +O \left( \frac{x E(x,y)}{\varphi(d) \log x} \right) ,
\end{equation}
where
\[ E(x,y) = y^{-1/21} \; \mbox{if} \; y \leq x^{2/3} \quad \mbox{and} \quad E(x,y) = x^{-1/6} \log x \; \mbox{if} \; y > x^{2/3} . \]
In this paper, we improve the bound \eqref{Libound} for $d=2$, 3, 4 and 6. \newline

In the case $d=2$, we prove the following result.

\begin{theorem} \label{resd2}
We have
\[ \frac{1}{y} \sum_{2 \leq a \leq y} P_{(a,2)}(x) - \frac{\pi(x)}{2} \ll \displaystyle{ \begin{cases}
x \sqrt{\frac{  \log x}{y}}, & \mbox{if} \; \frac{x}{\log x} \geq y , \\ \\
\sqrt{x} \log x , & \mbox{if} \; \frac{x}{\log x} < y .
\end{cases}   } \]
\end{theorem}

Additionally, for $d=3$. $4$ and $6$, we have the following.

\begin{theorem} \label{cubquarsex}
Let $\varepsilon >0$ be given.  If $d \in \{ 3, 4, 6\}$, then
\[ \frac{1}{y} \sum_{2 \leq a \leq y} P_{(a,d)}(x) - \frac{\pi(x;1,d)}{d} \ll (xy)^{\varepsilon} E(x,y), \]
where $E(x,y)$ is defined piece-wise in the following way \newline
\[
E(x,y) = \displaystyle{ \begin{cases}
 x^{1/2}, & \mbox{if} \; x \leq y^{3/5} , \\
x^{4/3}y^{-1/2} , & \mbox{if} \; y^{3/5} < x \leq y^{6/7}, \\
x^{3/4}, & \mbox{if} \; y^{6/7} < x \leq y^{6/5}, \\
x^{7/6}y^{-1/2}, & \mbox{if} \; y^{6/5} < x \leq y^{3/2} , \\
x^{5/6}, & \mbox{if} \; y^{3/2} < x \leq y^{9/5} , \\
x^{10/9}y^{-1/2}  , & \mbox{if} \; y^{9/5} < x \leq y^{108/55} , \\
x^{1/2}y^{7/10} , & \mbox{if} \; y^{108/55} < x \leq y^{11/5} , \\
x^{2/3}y^{1/3}  , & \mbox{if} \; y^{11/5} < x \leq y^{5/2} , \\
xy^{-1/2}  , & \mbox{if} \; y^{5/2} < x .
\end{cases}   }
\]
\end{theorem}

Note that the $O$-terms in Theorems~\ref{resd2} and \ref{cubquarsex} are smaller than that in \eqref{Libound} in all $xy$-ranges.  Finally, we shall prove the following smoothed version of Theorem~\ref{resd2}, assuming the truth of the generalized Riemann hypothesis.

\begin{theorem} \label{resd2smooth}
Let $x$, $Y>1$, $U= x^{1/8}Y^{1/4}$ and 
\[ D(Y) = \{ d : Y \leq d \leq 2Y, d \; \mbox{odd and square-free} \}  \]
and $\Phi_Y(t)$ be a non-negative smooth function supported on the interval $(1,2)$ satisfying the condition $\Phi_Y(t) = 1$ for $t \in (1+1/U, 2-1/U)$ and
\begin{equation} \label{Phiderivbound}
 \Phi_Y^{(j)}(t) \ll_j U^j 
 \end{equation}
for all integers $j \geq 0$.  Assuming the truth of the generalized Riemann hypothesis, we have
\begin{equation} \label{resd2smootheq}
\frac{1}{\# D(Y)} \sum_{(a,2)=1} \mu^2(a) P_{(8a,2)} (x) \Phi_Y \left( \frac{a}{Y} \right)  = \frac{\pi(x)}{2} + O \left( \log \log x  + \left( \frac{x^{7/8}}{Y^{1/4}}+\frac{x}{Y^{1/2}} \right) \log (xY)  \right)  .
\end{equation}
\end{theorem}

Note that the strength of Theorem~\ref{resd2smooth} is most prominent when $Y$ is large compared to $x$.

\section{Preliminaries}

We shall need the following lemma due to P\'olya and Vinogradov.

\begin{lemma} \label{polyvino}
Let $\chi(n)$ be a non-principal character modulo $q$ and $M \in \intz$ and $N \in \natn$.  Then we have
\[ \left| \sum_{M < n \leq M+N} \chi(n) \right| \leq 6 \sqrt{q} \log q. \]
\end{lemma}

\begin{proof}
This is Theorem 12.5 in \cite{HIEK}.
\end{proof}

The following mean-value theorem for quadratic character sums will be useful for us as well.

\begin{lemma} \label{jutest}
Let $\chi$ run over all real non-principal characters $\chi(n) = \left( \frac{D}{n} \right)$ with $|D| \leq X$.  We have
\[ \sum_{\chi} \left| \sum_{n \leq Y} \chi(n) \right|^2 \ll XY \log^2 X. \]
\end{lemma}

\begin{proof}
This is Lemma 5 in \cite{Jutila}
\end{proof}

\begin{lemma} \label{charsumprimeRH}
Assuming the truth of the generalized Riemann hypothesis for Dirichlet $L$-functions.  Then for any non-principal Dirichlet character $\chi \pmod{q}$, we have
\[ \sum_{ p \leq X} \chi(p) \ll X^{1/2} \log (qX). \]
\end{lemma}

\begin{proof}
This follows readily from \cite[Proposition 5.25]{HIEK}.
\end{proof}

\begin{lemma}[Large sieve inequality for certain fixed order characters]
Let $(a_m)_{m \in\natn}$ be an arbitrary sequence of complex number and $k\in \{ 3, 4, 6 \}$.  Then, for any $\varepsilon >0$,
\begin{equation} \label{lsineq}
\begin{split}
 \sum_{Q < q \leq 2Q} \ & \sideset{}{^{\star}} \sum_{\substack{\chi \bmod{q} \\ \chi^k = \chi_0, \chi \neq \chi_0}} \left| \ \sideset{}{^*} \sum_{M < m \leq 2M} a_m \chi(m) \right|^2 \\
 & \ll (QM)^{\varepsilon} \min \left\{ Q^{5/3}+M, Q^{4/3}+Q^{1/2}M, Q^{11/9}+Q^{2/3}M, Q+Q^{1/3}M^{5/3}+M^{12/5} \right\}  \sideset{}{^*} \sum_{M < m \leq 2M} |a_m|^2,
 \end{split}
 \end{equation}
 where the implied constant depends on $\varepsilon$, the star on the sum over $\chi$ restricts the sum to primitive characters and the asterisks attached to the sum over $m$ indicates that $m$ ranges over square-free integers.
\end{lemma}

\begin{proof}
For $k=4$ and $k=6$, the outer sums on the left-hand side of \eqref{lsineq} includes quadratic characters.  We use \cite[Theorem 1]{DRHB} to bound that part of the sum.  For other characters, we use the results in \cites{G&Zhao, B&Y}.  For $k=4$, this lemma is an improvement of \cite[Theorem 1.2]{G&Zhao} and the proof goes along the same line as Section 6 of \cite{G&Zhao}.  The only difference is, instead using \cite[Theorem 1.1]{G&Zhao}, one uses \cite[Theorem 1.3]{BGL} at the appropriate places.  For $k=3$ and $k=6$, this is \cite[Theorem 1.4]{B&Y} and \cite[Theorem 1.5]{B&Y}, respectively.
\end{proof}

\section{Proof of Theorem~\ref{resd2}}

By considering the group structure of $(\intz/p\intz)^*$, we easily conclude that $a^{(p-1)/2} \equiv 1 \pmod{p}$ if and only if $p \nmid a$ and $a$ is a square modulo $p$.   So 
\[ P_{(a,2)} (x) = \frac{1}{2} \sum_{2 < p \leq x} \left( \chi_0(a) + \left( \frac{a}{p} \right) \right) , \]
where $\chi_0$ is the principal character modulo $p$.  Therefore, we get
\begin{equation} \label{Ssplit}
S:=  \frac{1}{y} \sum_{2 \leq a \leq y} P_{(a,2)}(x)  = \frac{1}{2y} \sum_{2 < p \leq x} \sum_{a \leq y} \chi_0(a) + \frac{1}{2y} \sum_{2 < p \leq x} \sum_{a \leq y} \left( \frac{a}{p} \right) = S_1 + S_2,
 \end{equation}
say.  We easily have
\begin{equation} \label{s1est}
 S_1 = \frac{1}{2y} \sum_{2 < p \leq x} \left( [y] - \left[ \frac{y}{p} \right] \right) = \frac{1}{2y} \sum_{2 < p \leq x} \left( y - \frac{y}{p} + O(1) \right) = \frac{\pi(x)}{2} + O \left( \log \log x + \frac{\pi(x)}{y} \right) ,
 \end{equation}
where we have used the well-known bound
\begin{equation} \label{prepest}
 \sum_{p \leq x} \frac{1}{p} \ll \log \log x.
 \end{equation}
Now
\[  S_2 \ll \frac{1}{2y} \sum_{2 < p \leq x} \sum_{a \leq y} \left( \frac{a}{p} \right) \ll \frac{1}{2y} \sum_{2 < p \leq x} \left| \sum_{a \leq y} \left( \frac{a}{p} \right) \right| .  \]
Applying Lemma~\ref{polyvino}, we get that
\begin{equation} \label{s2est1}
S_2 \ll \frac{1}{2y} \sum_{p \leq x} \sqrt{p} \log p \ll \frac{x^{3/2}}{y} , 
\end{equation}
by the prime number theorem.  On the other hand, using Cauchy-Schwarz inequality, we conclude that
\[ \sum_{p \leq x} \left| \sum_{a \leq y} \left( \frac{a}{p} \right) \right| \leq \pi(x)^{1/2} \left( \sum_{p \leq x} \left| \sum_{a \leq y} \left( \frac{a}{p} \right) \right|^2 \right)^{1/2} \ll xy^{1/2} \log^{1/2} x , \]
where the last inequality comes from Lemma~\ref{jutest} and the prime number theorem.  Hence,
\begin{equation} \label{s2est2}
 S_2 \ll \frac{1}{y} x y^{1/2} \log^{1/2} x = \frac{x \log^{1/2} x}{y^{1/2}} .
 \end{equation}
Finally, combining \eqref{Ssplit}, \eqref{s1est}, \eqref{s2est1} and \eqref{s2est2}, we arrive at
\[ S = \frac{\pi(x)}{2} + O \left( \log \log x + \frac{\pi(x)}{y} + \min \left( \frac{x^{3/2}}{y} , \frac{x \log^{1/2} x}{y^{1/2}} \right) \right).  \]
The two terms in the minimum above are equal when $x/\log x = y$.  The theorem follows by an easy analysis.

\section{Proof of Theorem~\ref{cubquarsex}}

For $d=3$, $4$ and $6$, an argument similar to that in the previous section gives that
\[ \frac{1}{y} \sum_{a \leq y} P_{(a,d)} = \frac{1}{dy} \sum_{\substack{p \leq x \\ p \equiv 1 \bmod{d}}} \sum_{a \leq y} \chi_0(a) + \frac{1}{dy} \sum_{\substack{p \leq x \\ p \equiv 1 \bmod{d}}} \sum_{a \leq y} \sum_{\substack{\chi \bmod{p} \\ \chi^d = \chi_0, \ \chi \neq \chi_0}} \chi(a) =: S_1 + S_2, \]
say.  As before,
\begin{equation} \label{s1cubest}
 S_1 = \frac{\pi(x; 1,d)}{d} + O \left( \log \log x + \frac{\pi(x;1,d)}{y} \right) .
 \end{equation}
To estimate $S_2$, we note that for all $a \in \natn$, $a$ can be written uniquely as $l^2m$ with $l, m \in \natn$ and $m$ square-free.  Hence, using Cauchy's inequality,
\[ S_2^2 \ll \frac{\pi(x;1,d)}{y^2} \sum_{\substack{p \leq x \\ p \equiv 1 \bmod{d}}} \sum_{\substack{\chi \bmod{p} \\ \chi^d = \chi_0, \ \chi \neq \chi_0}} \left| \sideset{}{^*} \sum_{m \leq y} \left( \sum_{l \leq \sqrt{y/m}} \chi(l^2) \right) \chi(m) \right|^2 , \]
where the asterisks on the sum over $m$ indicates that $m$ runs ver square-free integers.  Now we will apply Lemma~\ref{lsineq} to $S_2$.   Aided by table in \cite[p. 897]{B&Y}, we get that
\begin{equation} \label{s2cubest}
 S_2 \ll (xy)^{\varepsilon} \sqrt{\pi(x;1,d)} \times \left\{ \begin{array}{cl} 1 & \mbox{if} \; x \leq y^{3/5} \\ x^{5/6}/y^{1/2} & \mbox{if} \;  y^{3/5} < x \leq y^{6/7} \\ x^{1/4} & \mbox{if} \; y^{6/7} < x \leq y^{6/5} \\ x^{2/3}/y^{1/2} & \mbox{if} \; y^{6/5} < x \leq y^{3/2} \\ x^{1/3} & \mbox{if} \; y^{3/2} < x \leq y^{9/5} \\ x^{11/18}/y^{1/2} & \mbox{if} \; y^{9/5} < x \leq y^{108/55} \\ y^{7/10} & \mbox{if} \; y^{108/55} < x < y^{11/5} \\ x^{1/6}y^{1/3} & \mbox{if} \; y^{11/5} < x \leq y^{5/2} \\ x^{1/2}/y^{1/2} & \mbox{if} \; y^{5/2} < x \end{array} \right.
 \end{equation}
Now a quick analysis will give that the $O$-terms in \eqref{s1cubest} can be absorbed into the bounds in \eqref{s2cubest} in all ranges.  The result thus follows.

\section{Proof of Theorem~\ref{resd2smooth}}

Set
\[ D(Y) = \{ d : Y \leq d \leq 2Y, d \; \mbox{odd and square-free} \} . \]
It is easy to see that
\begin{equation} \label{sizedy}
 \#D(Y) \sim \frac{4Y}{\pi^2} .
 \end{equation}
Now let $\Phi_Y(t)$ be a non-negative smooth function supported on the interval $(1,2)$ satisfying the conditions $\Phi_Y(t)=1$ for $t \in (1+1/U, 2-1/U)$ and such that 
\begin{equation} \label{Phidevbound}
\Phi_Y^{(j)}(t) \ll_j U^j
\end{equation}
 for all integers $j \geq 0$.  We need to find an asymptotic formula for 
\begin{equation} \label{Ssmoothsplit}
\begin{split}
 S &= \frac{1}{\#D(Y)} \sum_{(a,2)=1} \mu^2 (a) P_{(8a,2)} (x) \Phi_Y \left( \frac{a}{Y} \right) \\
 & = \frac{1}{2 \#D(Y)} \sum_{2 < p \leq x} \sum_{(a,2)=1} \mu^2 (a) \chi_0(8a) \Phi_Y \left( \frac{a}{Y} \right) + \frac{1}{2\#D(Y)} \sum_{2 < p \leq x} \sum_{(a,2)=1} \mu^2 (a) \left( \frac{8a}{p} \right) \Phi_Y \left( \frac{a}{Y} \right) .
 \end{split}
 \end{equation}
Let $S_1$ and $S_2$, respectively, denote the two sums above.  We shall study them separately.  \newline
 
Recalling that the support of $\Phi_Y$ lies inside the interval $(1,2)$ and the bound \eqref{Phiderivbound}, we have
\begin{equation} \label{S1split}
 S_1 = \frac{1}{2\# D(Y)} \sum_{2 < p \leq x} \left( \sum_{\substack{(a,2)=1 \\ Y \leq a \leq 2Y}} \mu^2 (a) \chi_0(8a) + O \left( \frac{Y}{U} \right) \right) = \frac{1}{2\# D(Y)} \sum_{2 < p \leq x} \sum_{\substack{(a,2p)=1 \\ Y \leq a \leq 2Y}} \mu^2 (a) + O \left( \frac{\pi(x)}{ U} \right)  . 
 \end{equation}
Using M\"obius inversion, we get 
\begin{equation} \label{mobsplit}
 \sum_{2 < p \leq x} \sum_{\substack{(a,2p)=1 \\ Y \leq a \leq 2Y}} \mu^2 (a) = \sum_{2 < p \leq x} \sum_{\substack{(a,2)=1 \\ Y \leq a \leq 2Y}} \mu^2 (a) \sum_{d|(a,p)} \mu(d) = \sum_{2 < p \leq x} \sum_{d|p} \mu(d) \sum_{\substack{(a,2)=1, \ d|a \\ Y \leq a \leq 2Y}} \mu^2 (a) .
 \end{equation}
Now it is easy to see that the above can be recast as
\begin{equation} \label{sfest}
 \sum_{2 < p \leq x} \sum_{\substack{Y \leq a \leq 2Y \\ (a,2)=1}} \mu^2(a) + O \left( \sum_{p \leq x} \sum_{\substack{ Y \leq a \leq 2Y \\ p|a}} 1 \right) = \pi(x) \# D(Y) + O \left( Y \sum_{p \leq x} \frac{1}{p} \right). 
 \end{equation}
Now combining \eqref{S1split}, \eqref{mobsplit}, \eqref{sfest} and \eqref{prepest}, we arrive at
\begin{equation} \label{S1smoothest}
S_1 = \frac{\pi(x)}{2} + O \left( \frac{Y}{\# D(Y)} \log \log x + \frac{\pi(x)}{U} \right) .
\end{equation}

The treatment of $S_2$ will be more complicated.  Let $z$ be a positive number to be chosen later.  We start with 
\[ \mu^2(a) = M_z(a) + R_z(a) , \; \mbox{where} \; M_z(a) = \sum_{\substack{ l^2 |a \\ l\leq z}} \mu(l) \; \mbox{and} \; R_z(a) = \sum_{\substack{l^2|a \\ l > z}} \mu(l) . \]
We split $S_2$ as 
\begin{equation} \label{S2smoothsplit}
 S_2 = S_{2,1} + S_{2,2} 
 \end{equation}
where
\[ S_{2,1} = \frac{1}{2\#D(Y)} \sum_{2 < p \leq x} \sum_{(a,2)=1} M_z(a) \left( \frac{8a}{p} \right) \Phi_Y \left( \frac{a}{Y} \right) \;
\mbox{and}
\; S_{2,2} = \frac{1}{2\#D(Y)} \sum_{2 < p \leq x} \sum_{(a,2)=1} R_z(a) \left( \frac{8a}{p} \right) \Phi_Y \left( \frac{a}{Y} \right) . \]
We first deal with $S_{2,2}$.  Interchanging summations gives us
\[ S_{2,2} \ll \frac{1}{\#D(Y)} \sum_{l > z} \sum_{\substack{m \geq 1 \\ (m,2)=1}} \Phi_Y \left( \frac{l^2 m}{Y} \right) \left| \sum_{2 < p \leq x}  \left( \frac{8l^2 m}{p} \right) \right| . \]
Using Lemma~\ref{charsumprimeRH}, we get that the above is
\begin{equation} \label{S22est}
 S_{2,2} \ll \frac{1}{\#D(Y)} \sum_{l > z} \sum_{m \geq 1} \Phi_Y \left( \frac{l^2 m}{Y} \right) x^{1/2} \log (l^2 m x) \ll \frac{\#D(Y)}{Y} \frac{x^{1/2} \log (xYz)}{z} .
 \end{equation}

  We now evaluate the inner sum of $S_{2,1}$ following a method of K. Soundararajan in \cite{sound1} by applying the Poisson summation
    formula to the sum over $d$. For all odd
    integers $k$ and all integers $m$, we introduce the Gauss-type
    sums
\begin{equation*}
\label{010}
    \tau_m(k) := \sum_{a \shortmod{k}}\left( \frac {a}{k} \right) e \left( \frac {am}{k} \right) =:
    \left( \frac {1+i}{2}+\left( \frac {-1}{k} \right)\frac {1-i}{2}\right) G_m(k).
\end{equation*}
We quote \cite[Lemma 2.3]{sound1} which determines $G_m(k)$.
\begin{lemma}
\label{lem1}
   If $(k_1,k_2)=1$ then $G_m(k_1k_2)=G_m(k_1)G_m(k_2)$. Suppose that $p^a$ is
   the largest power of $p$ dividing $m$ (put $a=\infty$ if $m=0$).
   Then for $b \geq 1$ we have
\begin{equation*}
\label{011}
    G_m(p^b)= \left\{\begin{array}{cl}
    0  & \mbox{if $b\leq a$ is odd}, \\
    \varphi(p^b) & \mbox{if $b\leq a$ is even},  \\
    -p^a  & \mbox{if $b=a+1$ is even}, \\
    (\frac {m/p^a}{p})p^a\sqrt{p}  & \mbox{if $b=a+1$ is odd}, \\
    0  & \mbox{if $b \geq a+2$}.
    \end{array}\right.
\end{equation*}
\end{lemma}

   For a Schwartz function $F$, let
\begin{equation} \label{tildedef}
   \tilde{F}(\xi)=\frac {1+i}{2}\hat{F}(\xi)+\frac
   {1-i}{2}\hat{F}(-\xi)=\int\limits^{\infty}_{-\infty}\left(\cos(2\pi \xi
   x)+\sin(2\pi \xi x) \right)F(x) \dif x.
\end{equation}
    We quote Lemma 2.6 of \cite{sound1} which determines the
   inner sum of $S_{2,1}$.
\begin{lemma}
\label{lem2}
   Let $\Phi$ be a non-negative, smooth function supported in
   $(1,2)$. For any odd integer $k$,
\begin{equation*}
\label{013}
  \sum_{\gcd(d,2)=1}M_z(d)\left( \frac {d}{k} \right)
    \Phi\left( \frac {d}{X} \right)=\frac {X}{2k}\left( \frac {2}{k} \right) \sum_{\substack {\alpha \leq z \\ (\alpha, 2k)=1}}\frac {\mu(\alpha)}{\alpha^2}
    \sum_m(-1)^mG_m(k)\tilde{\Phi}\left( \frac {mX}{2\alpha^2k} \right),
\end{equation*}
where $\tilde{\Phi}$ is as defined in \eqref{tildedef}.
\end{lemma}

Applying Lemmas~\ref{lem1} and~\ref{lem2}, we transform $S_{2,1}$ into
\begin{equation} \label{S21split}
\begin{split}
 S_{2,1} & = \frac{Y}{4\#D(Y)} \sum_{2<p\leq x} \frac{1}{p} \sum_{\substack{\alpha \leq z \\ (\alpha, 2p)=1}} \frac{\mu(\alpha)}{\alpha^2} \sum_m (-1)^m G_m(p) \tilde{\Phi} \left( \frac{mY}{2\alpha^2 p} \right) \\
 & = \frac{Y}{4\#D(Y)} \sum_{2 < p \leq x} \frac{1}{\sqrt{p}} \sum_{\substack{\alpha \leq z \\ (\alpha, 2p)=1}} \frac{\mu(\alpha)}{\alpha^2} \sum_{m \neq 0} (-1)^m \left( \frac{m}{p} \right) \tilde{\Phi} \left( \frac{mY}{2 \alpha^2 p} \right) = S_{m=\Box} + S_{m \neq \Box}.
 \end{split}
 \end{equation}
where $S_{m=\Box}$ is the part of $S_{2,1}$ with $m$ being a square and $S_{m\neq \Box}$ the complementary sum.  We shall also use the notations $m = \Box$ and $m \neq \Box$ to mean, respectively, $m$ is a square and $m$ is not a square. \newline

Now
\begin{equation} \label{Ssq}
S_{m=\Box} = \frac{Y}{2\#D(Y)} \sum_{2 < p \leq x} \frac{1}{\sqrt{p}} \sum_{\substack{\alpha \leq z \\ (\alpha, 2p)=1}} \frac{\mu(\alpha)}{\alpha^2} \sum_{\substack{m =1 \\ (m,p)=1}}^{\infty} (-1)^m \tilde{\Phi} \left( \frac{m^2Y}{2\alpha^2 p} \right) .
\end{equation}
Note that trivially estimating the integral in \eqref{tildedef} gives
\begin{equation*}
    \tilde{\Phi}(\xi)\ll 1,
\end{equation*}
and integration by parts yields
\begin{equation*}
\label{17}
    \tilde{\Phi}(\xi)=\frac {-1}{2\pi \xi}\left(\int\limits^{1+1/U}_{1}+\int\limits^{2}_{2-1/U} \right)
    \Phi'(x) \Bigl(\sin(2\pi \xi
   x)-\cos(2\pi \xi x) \Bigr ) \dif x \ll \frac 1{|\xi|}.
\end{equation*}
These bounds for $\tilde{\Phi}$ give
\[ \sum_{m=1}^{\infty} \left| \tilde{\Phi} \left( \frac{m^2Y}{2 \alpha p} \right) \right| \ll \sum_{m \leq (2 \alpha^2p/Y)^{1/2}} 1 + \sum_{m > (2 \alpha^2p/Y)^{1/2}} \frac{2\alpha^2 p}{m^2 Y} \ll \alpha \sqrt{\frac{p}{Y}} . \]
So inserting the above estimate into \eqref{Ssq}, we get
\begin{equation} \label{Ssqbound}
 S_{m=\Box} \ll \frac{Y^{1/2}x}{\#D(Y)} \log z .
 \end{equation}
The estimation of $S_{m\neq\Box}$ is more complicated.  Re-arranging the orders of summation and using $\left( \frac{4\alpha^2}{p} \right)$ to detect the condition $(2\alpha, p)=1$, we get
\[ S_{m\neq\Box} = \frac{Y}{4\#D(Y)} \sum_{\substack{\alpha \leq z \\ (\alpha , 2)=1}} \frac{\mu(\alpha)}{\alpha^2} \sum_{m \neq 0, \Box} (-1)^m \sum_{p \leq x} \left( \frac{4 \alpha^2 m}{p} \right) \frac{1}{\sqrt{p}} \tilde{\Phi} \left( \frac{mY}{2 \alpha^2 p} \right).  \]

Using integration by parts and Lemma~\ref{charsumprimeRH}, we get that the inner-most sum in the above displayed equation is
\begin{equation*}
\begin{split}
 \ll \log (\alpha ( |m|+2) x) \left| \tilde{\Phi} \left( \frac{mY}{2 \alpha^2 x} \right) \right| + \int\limits_1^x & \frac{\log ( \alpha (|m|+2) V)}{V} \left| \tilde{\Phi} \left( \frac{mY}{2 \alpha^2 V} \right) \right| \dif V \\
 & + \int\limits_1^x \frac{Y}{\alpha^2} \frac{\log ( \alpha (|m|+2) V)}{V^2} \left| m \tilde{\Phi}' \left( \frac{mY}{2 \alpha^2 V} \right) \right| \dif V .
 \end{split}
 \end{equation*}

Thus
\begin{equation} \label{Snosqsplit}
 S_{m\neq\Box}  \ll R_1 + R_2 + R_3 ,
 \end{equation}
where
\[ R_1 = \frac{Y}{\#D(Y)} \sum_{\alpha \leq z} \frac{1}{\alpha^2} \sum_m \log ( \alpha (|m|+2) X ) \left| \tilde{\Phi} \left( \frac{mY}{2 \alpha^2 x} \right) \right|  , \]
\[ R_2 = \frac{Y}{\#D(Y)} \sum_{\alpha \leq z} \frac{1}{\alpha^2} \int\limits_1^x \sum_m \frac{\log ( \alpha (|m|+2) V)}{V} \left| \tilde{\Phi} \left( \frac{mY}{2 \alpha^2 V} \right) \right| \dif V  \]
and
\[ R_3 = \frac{Y}{\#D(Y)} \sum_{\alpha \leq z} \frac{Y}{\alpha^4} \int\limits_1^x \sum_m \frac{\log ( \alpha (|m|+2) V)}{V^2} \left| m \tilde{\Phi}' \left( \frac{mY}{2 \alpha^2 V} \right) \right| \dif V . \]

Now observe that
\begin{equation} \label{msumbound}
\begin{split}
 \sum_m \log \left( \alpha ( |m|+2) x \right) \left| \tilde{\Phi} \left( \frac{mY}{2 \alpha^2 x} \right) \right| & \ll \sum_{m \leq 2 \alpha^2 x/Y} \log \left( \alpha ( |m|+2) x \right) + \sum_{m > 2\alpha^2 x/Y} \log (\alpha ( |m|+2) x) \frac{\alpha^4 x^2 U}{m^2 Y^2}  \\
& \ll \frac{\alpha^2 x}{Y} \log (\alpha x) + \frac{\log (\alpha x) \alpha^2 x}{Y} U \ll \frac{\log (\alpha x) \alpha^2 x}{Y} U 
 \end{split}
 \end{equation}
 and similarly, using \eqref{Phidevbound},
 \begin{equation} \label{msumbound2}
  \sum_m \log (\alpha ( |m|+2) V) \left| m \tilde{\Phi}' \left( \frac{mY}{2 \alpha^2 V} \right) \right|  \ll \log(\alpha V) \frac{\alpha^4 V^2 U^2}{Y^2} .
 \end{equation}
From \eqref{msumbound}, it follows that
\begin{equation} \label{R1bound}
 R_1 \ll \frac{xU}{\#D(Y)} \sum_{\alpha \leq z} \log ( \alpha x) \ll \frac{xU}{\#D(Y)} z \log (zx) \; \;  \mbox{and} \; \; R_2 \ll \frac{U}{\#D(Y)} \sum_{\alpha \leq z} \int\limits_1^x \log (\alpha V) \dif V \ll \frac{U}{\#D(Y)} xz \log (zx).
 \end{equation}
Using \eqref{msumbound2}
\begin{equation} \label{R3bound}
R_3 \ll \frac{U^2}{\#D(Y)} \sum_{\alpha \leq z} \int\limits_1^x \log (\alpha V) \dif V \ll \frac{U^2}{\#D(Y)} xz \log (zx) .
\end{equation}
Combining \eqref{Snosqsplit}, \eqref{R1bound} and \eqref{R3bound}, we arrive at
\[ S_{m \neq \Box} \ll \frac{xz}{\#D(Y)} U^2 \log (zx) . \]
From this, together with \eqref{Ssqbound} and \eqref{S21split}, it follows
\[ S_{2,1} \ll \frac{Y^{1/2}x}{\#D(Y)} \log z + \frac{xz}{\#D(Y)} U^2 \log (zx) . \]
So from the above, together with \eqref{S22est}, \eqref{S2smoothsplit}, \eqref{S1smoothest}, \eqref{Ssmoothsplit} and \eqref{sizedy}, we get
\begin{equation} \label{Sest1}
 S = \frac{\pi(x)}{2} + O \left( \log \log x + \frac{\pi(x)}{ U}  + \frac{ x \log z}{Y^{1/2}} + \frac{xz}{Y} U^2 \log (zx) + \frac{x^{1/2} \log (xYz)}{z}  \right) .
 \end{equation}
Now it still remains to optimize $z$.  To this end, we set
\[ z = \frac{Y^{1/2}}{Ux^{1/4} .} \]
So \eqref{Sest1} becomes
\[ S = \frac{\pi(x)}{2} + O \left( \log \log x + \frac{\pi(x)}{U}  + \frac{ x \log Y}{Y^{1/2}} + \frac{Ux^{3/4}}{Y^{1/2}} \log (xY)  \right)  . \]
Now recall $U= x^{1/8}Y^{1/4}$ and $\pi(x) \sim x/\log x$.  The proof of Theorem~\ref{resd2smooth} is complete. \newline

\noindent{\bf Acknowledgments.} P. G. is supported in part by NSFC grants 11371043 and and 11871082 and L. Z. by the FRG grant PS43707 and the Faculty Silverstar Fund PS49334.  Parts of this work were done when P. G. visited the University of New South Wales (UNSW) in June 2017. He wishes to thank UNSW for the invitation, financial support and warm hospitality during his pleasant stay.  Finally, both authors would like to thank the anonymous referee for his/her careful reading of the manuscript and many helpful comments.

\bibliography{biblio}
\bibliographystyle{amsxport}

\vspace*{.5cm}

\noindent\begin{tabular}{p{8cm}p{8cm}}
School of Mathematics and Systems Science & School of Mathematics and Statistics \\
Beihang University & University of New South Wales \\
Beijing 100191 China & Sydney NSW 2052 Austrlia \\
Email: {\tt penggao@buaa.edu.cn} & Email: {\tt l.zhao@unsw.edu.au} \\
\end{tabular}

\end{document}